\documentclass[multi]{cambridge7A}
\usepackage[UKenglish]{babel}
\usepackage[longnamesfirst,sectionbib]{natbib}
\usepackage{chapterbib}
\usepackage{amsmath,amssymb,amsthm,amsfonts}
\usepackage{enumerate,url}
\usepackage{pictexwd}

\setcitestyle{authoryear,round,semicolon}

\providecommand*\Index[1]{#1\index{#1}}
\providecommand*\undex[1]{} 
\providecommand*\Undex[1]{#1} 

\theoremstyle{plain}
\newtheorem{theorem}{Theorem}[section]
\newtheorem{lemma}[theorem]{Lemma}

\allowdisplaybreaks[1]

\mathchardef\cd="201 
\newcommand{\N}{\mathbb{N}}
\newcommand{\R}{\mathbb{R}}
\newcommand{\oh}{\mathrm{o}}
\newcommand{\one}{\mathbf{1}}
\newcommand{\var}{\mathop{\mathrm{var}}}%
\newcommand{\Geom}{\mathop{\mathrm{Geom}}}
\newcommand{\abs}[1]{\lvert#1\rvert}

\begin{document}
\alphafootnotes
\author[C. M. Goldie, R. Cornish and C. L. Rob\-inson]{C. M.
  Goldie\footnotemark, R. Cornish\footnotemark\ and C. L.
  Robinson\footnotemark }
\chapter[Coupon collecting and assessment]{Applying coupon-collecting theory to
  computer-aided assessments}
\footnotetext[1]{Mathematics Department, Mantell Building, University of
  Sussex, Brighton BN1 9RF; C.M.Goldie@sussex.ac.uk}
\footnotetext[2]{Department of Social Medicine, University of Bristol, Canynge
  Hall, 39 Whatley Road, Bristol BS8 2PS; R.Cornish@bristol.ac.uk}
\footnotetext[3]{Mathematics Education Centre, Schofield Building,
  Loughborough University, Loughborough LE11 3TU;
  C.L.Robinson@lboro.ac.uk}
\arabicfootnotes
\contributor{Charles M. Goldie \affiliation{University of Sussex}}
\contributor{Rosie Cornish \affiliation{University of Bristol}}
\contributor{Carol L. Robinson \affiliation{Loughborough University}}
\renewcommand\thesection{\arabic{section}}
\numberwithin{equation}{section}
\renewcommand\theequation{\thesection.\arabic{equation}}
\numberwithin{figure}{section}
\renewcommand\thefigure{\thesection.\arabic{figure}}
\numberwithin{table}{section}
\renewcommand\thetable{\thesection.\arabic{table}}

\begin{abstract}Computer-based tests with randomly generated questions allow
a large number of different tests to be generated. Given a fixed number of
alternatives for each question, the number of tests that need to be generated
before all possible questions have appeared is surprisingly low.\end{abstract}

\subparagraph{AMS subject classification (MSC2010)}60G70, 60K99

\section{Introduction}\label{s:intro}
The use of computer-based tests\index{computer based test@computer-based test|(} in which questions are randomly
generated in some way provides a means whereby a large number of different
tests can be generated; many universities currently use such tests as part of
the student assessment process. In this paper we present findings that
illustrate that, although the number of different possible tests is high and
grows very rapidly as the number of alternatives for each question increases,
the average number of tests that need to be generated before all possible
questions have appeared at least once is surprisingly low. We presented
preliminary findings along these lines in \cite{CGR}.

A computer-based test consists of $q$
questions, each (independently) selected at random from a separate bank of $a$
alternatives. Let $N_q$ be the number of tests one needs to generate in order
to see all the $aq$ questions in the $q$ question banks\index{question bank} at least once. We are
interested in how, for fixed $a$, the random variable $N_q$ grows with the
number of questions $q$ in the test. Typically, $a$ might be 10---i.e.\ each
question might have a bank of 10 alternatives---but we shall allow any value
of $a$, and give numerical results for $a=20$ and $a=5$ as well as for $a=10$.

\section{Coupon collecting}\label{s:coupon}\index{coupon collecting|(}
In the case $q=1$, i.e.\ a one-question test, we re-notate $N_q$ as
$Y$, and observe that we have an equivalent to the classic coupon-collector
problem: your favourite cereal has a coupon in each packet, and there are $a$
alternative types of coupon. $Y$ is the number of packets you have to buy in
order to get at least one coupon of each of the $a$ types. The
coupon-collector problem has been much studied; see e.g.\ \cite[p.~55]{GS}.

We can write $Y$ as
$$ Y = Y_1 + Y_2 + \cdots + Y_a$$
where each $Y_i$ is the number of cereal packets you must buy in order to
acquire a new type of coupon, when you already have $i-1$ types in your
collection. Thus $Y_1=1$, $Y_2$ is the number of further packets you find you
need to gain a second type, and so on. The random variables $Y_1$, \dots,
$Y_a$ are mutually independent. For the distribution of $Y_k$, clearly
$$P(Y_k=y)=\frac{a-k+1}a\left(\frac{k-1}a\right)^{y-1}\qquad(y=1,2,\ldots).$$
We say that $X\sim \Geom(p)$, or $X$ has a geometric distribution with
parameter $p$, if $P(X=x)=p(1-p)^{x-1}$ for $x=1$, 2, \ldots. Thus $Y_k\sim
\Geom\bigl((a-k+1)/a\bigr)$. As the $\Geom(p)$ distribution has expectation
$1/p$ it follows that
$$EY=\sum_{k=1}^a EY_k=\sum_{k=1}^a\frac a{a-k+1}=a\sum_{k=1}^a\frac 1k.$$
For different values of $a$ we therefore have the following.

{\small
\begin{center}
\begin{tabular}{c|cccc}
$a$ & 5 & 10 & 15 & 20 \\ 
$EY$ & $11\cd42$ & $29\cd29$ & $49\cd77$ & $71\cd96$ \\ 
\end{tabular}
\end{center}
}%

\noindent In other words, if there are 10 coupons to collect then an average
of 29 packets of cereal would have to be bought in order to obtain all 10 of
these coupons. In the context of computer-based tests, if a test had one
question selected at random from a bank of 10 alternatives, an average of 29
tests would need to be generated in order to see all the questions at least
once.

To apply the theory to tests with more than one question we will also need an
explicit expression for $P(Y>y)$. To revert to the language of coupons in
cereal packets, let us number the coupon types 1, 2, \ldots, $a$, and let
$A_i$ be the event that type $i$ does not occur in the first $y$ cereal
packets bought. The event that $Y>y$ is then the union of the events $A_1$,
$A_2$, \ldots, $A_a$. So by the \Index{inclusion-exclusion} formula,
\begin{align*}
P(Y>y)&=P\left(\bigcup_{i=1}^a A_i\right)\\
  &=\sum_{i=1}^a P(A_i)-\sum_{i<j}P(A_i\cap A_j)
    +\sum_{i<j<k} P(A_i\cap A_j\cap A_k)-\cdots\\
  &\qquad{}+(-1)^{a+1}P(A_1\cap\cdots\cap A_a).
\end{align*}
Obviously $P(A_i)=(1-1/a)^y$ for each $i$. For distinct $i$ and $j$, $A_i\cap
A_j$ is the event that a particular two of the $a$ coupon types do not occur
in the first $y$ purchases, so has probability $(1-2/a)^y$. Similarly $A_i\cap
A_j\cap A_k$, for distinct $i$, $j$ and $k$, has probability $(1-3/a)^y$, and
so on. We conclude that
\begin{equation}P(Y>y)=\sum_{k=1}^a(-1)^{k+1}\binom ak\Bigl(1-\frac ka\Bigr)^y
\label{e:Yy}\end{equation}
(when $y>0$ the final term of the sum is zero). Let $F$ be the distribution
function for $Y$; thus the above is equivalent to
\begin{equation}F(y):=P(Y\le y)
  =\sum_{k=0}^a(-1)^k\binom ak\Bigl(1-\frac ka\Bigr)^y
  \quad(y=a,a+1,a+2,\dots).\label{e:Fy}\end{equation}
This is a classical formula for the probability that all cells are occupied\undex{occupancy}
when $y$ balls are distributed at random among $a$ cells\index{balls in cells};
cf.\ \cite[(11.11)]{Fe}\index{Feller, W.}. The right-hand side of (\ref{e:Fy}) has value 0 when
$y=0$, 1, \dots, $a-1$.

\section{How many tests?}\label{s:howmany}
We return to the initial question. We have a test containing
$q$ questions, each selected at random from a bank of $a$ alternatives. $N_q$
is defined to be the number of tests that need to be generated in order to see
all possible $aq$ questions at least once.

For question $j$ of the test, let $Y_j$ be the number of tests needed to see
all the $a$ alternatives in its \Index{question bank}. The random variables $Y_1$,
$Y_2$, \ldots, $Y_q$ are mutually independent, each distributed as the $Y$ of
the previous section, and $N_q$ is their maximum:
$$N_q=\max\{Y_1, Y_2,\ldots,Y_q\}.$$
We thus have
\begin{align}
  EN_q&= \sum_{n=0}^\infty P(N_q>n)\notag \\
  &= \sum_{n=0}^\infty \bigl(1-P(N_q\leq n)\bigr)\notag \\
  &= \sum_{n=0}^\infty \left(1-\prod_{j=1}^q P(Y_j\leq n)\right)\notag \\
  &= \sum_{n=0}^\infty \left(1-\prod_{j=1}^q
    \bigl(1-P(Y_j>n)\bigr)\right)\notag \\
  &= \sum_{n=0}^\infty \left(1-\bigl(1-P(Y>n)\bigr)^q\right).\label{e:ENq}
\end{align}
This can be reduced to a finite sum as follows.
\begin{align*}
  EN_q&=\sum_{n=0}^\infty\sum_{m=1}^q(-1)^{m+1}\binom qm\bigl(P(Y>n)\bigr)^m\\
  &=\sum_{n=0}^\infty\sum_{m=1}^q(-1)^{m+1}\binom qm
 \sum_{j_1=1}^a(-1)^{j_1+1}\binom a{j_1} \left(1-\frac{j_1}{a}\right)^n
   \cdots\\
 &\qquad{}\cdots
 \sum_{j_m=1}^a(-1)^{j_m+1}\binom a{j_m}\left(1-\frac{j_m}{a}\right)^n\\
  &= -\sum_{m=1}^q\binom qm\sum_{j_1=1}^a\cdots\sum_{j_m=1}^a
    (-1)^{j_1+\cdots+j_m}\\
  &\qquad \qquad{}\binom a{j_1}\cdots\binom a{j_m}
  \sum_{n=0}^\infty\left(\prod_{i=1}^m\left(1-\frac{j_i}{a}\right)\right)^n\\
  &= -\sum_{m=1}^q\binom qm\sum_{j_1=1}^a\cdots\sum_{j_m=1}^a
    \frac{(-1)^{j_1+\cdots+j_m}\binom a{j_1}\cdots\binom a{j_m}}
  {1-\prod_{i=1}^m(1-j_i/a)}.
\end{align*}
This, though, is not well suited to computation, and we have used
(\ref{e:ENq}) for the numerical results below.


\subparagraph{Note}The way in which CMG got involved in writing this paper
was through chancing on a query posted by RC on \Index{Allstat}, a UK-based electronic
mailing list, asking how to calculate the expected number of tests a student
would need to access in order to see the complete bank of questions. CMG
immediately recognised the query as a form of coupon-collecting problem, but
not quite in standard form. What he should have done then was to think and
calculate, following Littlewood's famous
advice \cite[p.~93]{LM}\index{Littlewood, J. E.}
\begin{quote}
``It is of course good policy, and I have often practised it, to begin without
going too much into the existing literature''.
\end{quote}
What he actually did was to seek previous work using
\index{Google@{\it Google}}{\it Google}. With customary
speed and accuracy, \emph{Google} produced a list with
\cite{AR}\index{Ross, S. M.}\index{Adler, I.} in position
6. Knowing that Sheldon Ross is unbeatable at combinatorial probability
problems, CMG looked up this paper---and was thoroughly led astray. The paper
does indeed treat our problem and is an excellent paper, but it is much more
general than we needed and sets up a structure that obscures the relatively
simple nature of what we needed for this problem. It was better to work the
above out from first principles.\index{coupon collecting|)}


\section{Asymptotics}\label{s:asymp}
We employ Extreme-Value Theory (EVT)\index{extreme-value theory (EVT)} to investigate the random
variable $N_q$ as the number of questions $q$ becomes large, the number $a$ of
alternatives per question staying fixed. It turns out we are in a case
identified by C.~W. Anderson\index{Anderson, C. W.} in 1970, where a limit fails to exist but there
are close bounds above and below. Thus despite the absence of a limit we gain
asymptotic results of some precision.

The relevant extreme-value distribution will be the \emph{Gumbel
distribution}\index{Gumbel, E. J.!Gumbel distribution|(}, with (cumulative) distribution function
$\Lambda(x)=\exp(-e^{-x})$ for all $x\in\R$; write $Z$ for a random variable
with the Gumbel distribution.

Throughout this section $a\ge2$ is an integer, and we set
$\alpha:=\log(a/(a-1))>0$. Proofs of the results in this section are in
\S\ref{s:pfs}.

A first goal of EVT for the random variables $N_q$ would be
to find a \Undex{norming} sequence $a_q>0$ and a \Undex{centring} sequence $b_q$ such that
$(N_q-b_q)/a_q$ has a limit distribution as $q\to\infty$.

\begin{theorem}\label{t:noconv}There do not exist sequences $a_q>0$ and $b_q$
such that $(N_q-b_q)/a_q$ has a non-degenerate limit distribution as
$q\to\infty$. However, with $b_q:=\alpha^{-1}\log(aq)$ we have for all $x\in\R$
that
\begin{multline}\Lambda(\alpha(x-1))=\liminf_{q\to\infty}P(N_q-b_q\le x)\\
  \le\limsup_{q\to\infty}P(N_q-b_q\le x)=\Lambda(\alpha x).
  \label{e:Lambdabounds}\end{multline}
Thus $N_q-b_q$, in distribution, is asymptotically between $\alpha^{-1}Z$ and
$1+\alpha^{-1}Z$, with $Z$ Gumbel, and these \Index{distributional bounds} are sharp.

To describe the local behaviour\index{extreme-value theory (EVT)!local EVT}, let $\lfloor x\rfloor$ denote the integer
part of $x$, $\{x\}:=x-\lfloor x\rfloor$ the fractional part, and let $\lceil
x\rceil:=\lfloor x\rfloor+1$. Then for each integer $n$,
\begin{multline}P(N_q-\lceil b_q\rceil=n)
  -\Lambda\bigl(\alpha(n+1-\{b_q\})\bigr)
  +\Lambda\bigl(\alpha(n-\{b_q\})\bigr)\to0\\
  \mbox{as $q\to\infty$.}\label{e:local}\end{multline}
\end{theorem}

We remark that the Gumbel distribution\index{Gumbel, E. J.!Gumbel distribution|)} has mean $\gamma\bumpeq0\cd5772$, the
Euler--Mascheroni constant\index{Euler, L.!Euler Mascheroni constant@Euler--Mascheroni constant}, and
variance $\pi^2/6$. Its distribution tails\index{tail behaviour}
decay exponentially or better: $\lim_{x\to\infty}e^x(1-\Lambda(x))=1$ and
$\lim_{x\to-\infty}e^{-x}\break\Lambda(x)=0$. We use these facts below. We first
extend the above \index{stochastically bounded}stochastic boundedness of the sequence $(N_q-b_q)$ to
$L^p$-boundedness\index{Lp bounded@$L^p$-bounded} for all $p$. For the rest of the paper we set
$b_q:=\alpha^{-1}\log(aq)$ and $R_q:=N_q-b_q$.

\begin{theorem}\label{t:Lp}For each $p\ge1$,
$\sup_{q\in\N}E(\abs{R_q}^p)<\infty$.
\end{theorem}

Theorem \ref{t:Lp} implies that the distributional asymptotics of Theorem
\ref{t:noconv} will extend to give asymptotic bounds on
moments\index{moment!bound}. Moment convergence in
EVT\index{extreme-value theory (EVT)!moment convergence in} is treated in
\cite[\S2.1]{Res87}\index{Resnick, S. I.}, and we use some
of the ideas from the proofs there in proving the results below.

\begin{theorem}\label{t:L1}
\begin{multline*}\frac{\gamma+\log a}\alpha
  \le\liminf_{q\to\infty}\Bigl(EN_q-\frac{\log q}\alpha\Bigr)\\
  \le\limsup_{q\to\infty}\Bigl(EN_q-\frac{\log q}\alpha\Bigr)
  \le\frac{\gamma+\log a}\alpha+1.\end{multline*}
\end{theorem}

By similar methods one may obtain bounds on higher moments. We content
ourselves with those on the second moment, leading to good bounds on $\var
N_q$, the variance of $N_q$.\index{variance bound|(}

\begin{lemma}\label{l:Rq2}
\begin{multline}E\bigl((1+\alpha^{-1}Z)^2\one_{1+\alpha^{-1}Z\le0}
  +(\alpha^{-1}Z)^2\one_{Z>0}\bigr)\le\liminf_{q\to\infty}E(R_q^2)\\
  \le\limsup_{q\to\infty}E(R_q^2)\le E\bigl((\alpha^{-1}Z)^2\one_{Z\le0}
  +(1+\alpha^{-1}Z)^2\one_{1+\alpha^{-1}Z>0}\bigr).\label{e:Rq2}\end{multline}
\end{lemma}

\begin{theorem}\label{t:Nqvar}
$$\limsup_{q\to\infty}\Bigl|\var N_q-\frac{\pi^2}{6\alpha^2}\Bigr|
  \le\theta(\alpha)+1-e^{-1}
    +\frac{2\bigl(\gamma+E_1(1)\bigr)}\alpha,$$
where
$\theta(\alpha)=E\bigl((1+\alpha^{-1}Z)^2\one_{0<1+\alpha^{-1}Z\le1}\bigr)$
satisfies $0<\theta(\alpha)<1$, and
$E_1(1)=\int_1^\infty t^{-1}e^{-t}\,dt\bumpeq0\cd2194$.
\end{theorem}\index{variance bound|)}

Here, $E_1(1)$ is a value of the \Index{exponential integral} (cf.\ \cite[\S5.1]{AS})
$E_n(x)=\int_1^\infty t^{-n}e^{-xt}\,dt$.

\section{Proofs for \S\ref{s:asymp}}\label{s:pfs}

\begin{proof}[Proof of Theorem \ref{t:noconv}]In (\ref{e:Yy}) the $k=1$ term
dominates for large $y$, so
\begin{equation}P(Y>y)=a(1-1/a)^y(1+\oh(1))\label{e:Ftail}\end{equation}
as $y\to\infty$ through integer values. As noted
in \cite[\S1]{And70}\index{Anderson, C. W.|(}, the
fact that the integer-valued random variable $Y$ has
$$\frac{P(Y>y)}{P(Y>y+1)}\to\frac a{a-1}>1\quad\mbox{as $y\to\infty$}$$
prevents it from belonging to the `\Index{domain of attraction}' for maxima of any
\Index{extreme-value distribution}, and so no non-trivial limit distribution for
$(N_q-b_q)/a_q$, for any choices of $a_q$ and $b_q$, can exist.

For the rest of the proof, $b_q:=\alpha^{-1}\log(aq)$. Via the definition of
$\alpha$, (\ref{e:Ftail}) gives that $F(y)=1-ae^{-\alpha y}(1+\oh(1))$ as
$y\to\infty$ through integer values. So for each fixed $x\in\R$,
\begin{equation}
P(N_q-b_q\le x)=F^q(\lfloor x+b_q\rfloor)
  =\bigl(1-ae^{-\alpha\lfloor x+b_q\rfloor}(1+\oh(1))\bigr)^q\label{e:Nqbq}
\end{equation}
as $q\to\infty$. Then
\begin{align*}P(N_q-b_q\le x)&\le\bigl(1-ae^{-\alpha(x+b_q)}(1+\oh(1))\bigr)^q\\
  &=\left(1-\frac{e^{-\alpha x}(1+\oh(1))}q\right)^q
    \to\Lambda(\alpha x)\quad\mbox{as $q\to\infty$}.
\end{align*}
With $x\in\R$ still fixed we define the sequence $\bigl(q(k)\bigr)_{k=1}^\infty$
to be those $q$ for which the interval $(x+b_{q-1},x+b_q]$ contains one or
more integers, i.e.\ for which $x+b_{q-1}<\lfloor x+b_q\rfloor$. Since
$b_q\to\infty$ this is an infinite sequence, and since $b_{q+1}-b_q\to0$ we
have $x+b_{q(k)}-\lfloor x+b_{q(k)}\rfloor\to0$ as $k\to\infty$, whence with
(\ref{e:Nqbq}) we conclude that $P(N_{q(k)}-b_{q(k)}\le x)\to\Lambda(\alpha
x)$ as $k\to\infty$. Thus $\limsup_{q\to\infty}P(N_q-b_q\le
x)=\Lambda(\alpha x)$.

For the limit inferior,
\begin{align*}P(N_q-b_q\le x)&\ge\bigl(1-ae^{-\alpha(x-1+b_q)}(1+\oh(1))\bigr)^q\\
  &=\left(1-\frac{e^{-\alpha(x-1)}(1+\oh(1))}q\right)^q\\
  &\to\Lambda(\alpha(x-1))\quad\mbox{as $q\to\infty$}.
\end{align*}
With the same sequence $\bigl(q(k)\bigr)$ as above, note that
$x+b_{q(k)-1}-\lfloor x+b_{q(k)-1}\rfloor\to1$ as $k\to\infty$, so
\begin{align*}P(N_{q(k)-1}-b_{q(k)-1}\le x)
  &=\bigl(1-ae^{-\alpha\lfloor x+b_{q(k)-1}\rfloor}(1+\oh(1))\bigr)^{q(k)-1}\\
  &=\bigl(1-ae^{-\alpha(x+b_{q(k)-1}-1)}(1+\oh(1))\bigr)^{q(k)-1}
\end{align*}
by (\ref{e:Nqbq}). The right-hand side converges to
$\Lambda(\alpha(x-1))$. Thus $\liminf_{q\to\infty}\allowbreak P(N_q-b_q\le
x)=\Lambda(\alpha(x-1))$. This establishes (\ref{e:Lambdabounds}).

The extension to local behaviour\index{extreme-value theory (EVT)!local EVT} is due to \cite{And80}\index{Anderson, C. W.|)}. To gain the
conclusion as we formulate it, (\ref{e:local}), we may argue directly: fix an
integer $n$ and start from
$$P(N_q-\lceil b_q\rceil\le n)=F^q(n+\lceil b_q\rceil)
  =\bigl(1-ae^{-\alpha(n+\lceil b_q\rceil)}(1+\oh(1))\bigr)^q$$
as $q\to\infty$. Now
$$ae^{-\alpha(n+\lceil b_q\rceil)}=\frac{e^{-\alpha(n+\lceil b_q\rceil-b_q)}}q
  =\frac{e^{-\alpha(n+1-\{b_q\})}}q,$$
and as the convergence in $(1-c/q)^q\to e^{-c}$ is locally uniform in $c$ we
deduce that
$$P(N_q-\lceil b_q\rceil\le n)-\Lambda\bigl(\alpha(n+1-\{b_q\})\bigr)\to0
  \quad\mbox{as $q\to\infty$.}$$
Subtract from this the corresponding formula with $n$ replaced by $n-1$, and
(\ref{e:local}) follows.
\end{proof}

For the next result we need a uniform bound on expressions of the form
$1-(1-u/n)^n$:

\begin{lemma}\label{l:e} For any $u_0>0$ there exists a positive integer
  $n_1=n_1(u_0)$ such that for $n\ge n_1$ and $0\le u\le u_0$,
$$1-\Bigl(1-\frac un\Bigr)^n\le2u.$$
\end{lemma}

\begin{proof} There exists $t_0>0$ (its value is about $0\cd7968$) such that
  $\log(1-t_0)=-2t_0$, so $\log(1-t)\ge-2t$ for $0\le t\le t_0$. Take $n_1\ge
  u_0/t_0$, then $1-(1-u/n)^n\le1-e^{-2u}$ for $n\ge n_1$ and $0\le u\le u_0$,
  and as $1-e^{-2u}\le2u$ the result follows.
\end{proof}

\begin{proof}[Proof of Theorem \ref{t:Lp}] We write $\one_T:=1$ if statement
$T$ is true, $\one_T:=0$ if $T$ is false. Fix $n\in\N$. The distribution of
$N_q$ is such that $E(R_q^{2n})<\infty$ for all $q$. We prove that
$\sup_{q\in\N}E(R_q^{2n})<\infty$. Now
$$E(R_q^{2n})=\int_{(-\infty,0]}x^{2n}\,dP(R_q\le x)
  -\int_{(0,\infty)}x^{2n}\,dP(R_q>x),$$
and so, on integrating by parts,
\begin{align*}E(R_q^{2n})
  &=-2n\int_{-\infty}^0 x^{2n-1}P(R_q\le x)\,dx
    +2n\int_0^\infty x^{2n-1}P(R_q>x)\,dx\\
  &=:A+B,\end{align*}
say.

In (\ref{e:Yy}) the right-hand
side is asymptotic to its first term, $ae^{-\alpha y}$. There exists
$y_0$ such that for real $y\ge y_0$ (not just integer $y$),
$P(Y>y)\le2ae^{-\alpha(y-1)}$. So for $x\ge0$ and $q\ge a^{-1}e^{\alpha y_0}$,
$$P(Y>x+b_q)\le2ae^{-\alpha(x+b_q-1)}=\frac2q\, e^{\alpha-\alpha x},$$
and hence
$$P(R_q>x)=1-\bigl(1-P(Y>x+b_q)\bigr)^q
  \le1-\biggl(1-\frac2q\, e^{\alpha-\alpha x}\biggr)^q.$$
Now apply Lemma \ref{l:e}. It follows that there exists $q_1$ such that for
$q\ge q_1$ and $x\ge0$,
$$P(R_q>x)\le4e^{\alpha-\alpha x}.$$
Therefore, for $q\ge q_1$,
$$B=2n\int_0^\infty x^{2n-1}P(R_q>x)\,dx
  \le8n\int_0^\infty x^{2n-1}e^{\alpha-\alpha x}\,dx<\infty.$$

It remains to bound $A$. Returning again to (\ref{e:Yy}),
observe that we may find $y_1$ so that $P(Y>y)\ge \frac12 ae^{-\alpha y}$ for
all real $y\ge y_1$. Therefore for $x\ge y_1-b_q$ we have
$$P(Y>x+b_q)\ge\frac12 ae^{-\alpha(x+b_q)}=\frac1{2q}e^{-\alpha x},$$
and so
\begin{align}P(R_q\le x)&=\bigl(1-P(Y>x+b_q)\bigr)^q\notag \\
  &\le\exp\bigl(-qP(Y>x+b_q)\bigr)\notag \\
  &\le\exp\Bigl(-\frac12 e^{-\alpha x}\Bigr)\quad\mbox{for $x\ge y_1-b_q$.}
\label{e:negxbound}\end{align}

In $A=-2n\int_{-\infty}^0 x^{2n-1}P(R_q\le x)\,dx$, the lower endpoint of the
interval of integration may be taken to be $-b_q$, as the integrand vanishes
below this point, and we then choose further to split the integral to obtain
\begin{align*}A&=-2n\int_{y_1-b_q}^0 x^{2n-1}P(R_q\le x)\,dx
    -2n\int_{-b_q}^{y_1-b_q} x^{2n-1}P(R_q\le x)\,dx\\
  &=:A_1+A_2,\end{align*}
say. If we take $q$ so large that $b_q>y_1$, (\ref{e:negxbound}) gives
\begin{align*}A_1
  &\le-\int_{y_1-b_q}^0 x^{2n-1}\exp\Bigl(-\frac12 e^{-\alpha x}\Bigr)\,dx\\
  &<-2n\int_{-\infty}^0 x^{2n-1}\exp\Bigl(-\frac12 e^{-\alpha x}\Bigr)\,dx
  <\infty.\end{align*}
Finally,
\begin{align*}A_2&=-2n\int_{b_q}^{y_1-b_q}x^{2n-1}F^q(x+b_q)\,dx\\
  &=-2n\int_0^{y_1} (u-b_q)^{2n-1}F^q(u)\,du\\
  &\le2ny_1b_q^{2n-1}F^q(y_1)\\
  &=2ny_1\Bigl(\frac{\log(aq)}\alpha \Bigr)^{2n-1}F^q(y_1).\end{align*}
This tends to 0 as $q\to\infty$, because $0<F(y_1)<1$.

We have shown that $\limsup_{q\to\infty}E(R_q^{2n})<\infty$, so
$\sup_{q\in\N}E(R_q^{2n})<\infty$ as claimed, and the result follows.
\end{proof}

Before proving Theorem \ref{t:L1} we note that (\ref{e:Lambdabounds}) says
that for each $x\in\R$,
\begin{equation}\Lambda(\alpha(x-1))=\liminf_{q\to\infty}P(R_q\le x)
  \le\limsup_{q\to\infty}P(R_q\le x)=\Lambda(\alpha x),
  \label{e:Rqbounds}\end{equation}
and that what we have to prove is
\begin{equation}E(\alpha^{-1}Z)\le\liminf_{q\to\infty}ER_q
  \le\limsup_{q\to\infty}ER_q
  \le E(1+\alpha^{-1}Z).\label{e:ERqbounds}\end{equation}
We use (\ref{e:Rqbounds}) mostly in the form
\begin{multline}P(\alpha^{-1}Z>x)=\liminf_{q\to\infty}P(R_q>x)\\
  \le\limsup_{q\to\infty}P(R_q>x)=P(1+\alpha^{-1}Z>x),
  \label{e:Rqtail}\end{multline}
obtained by subtracting each component from 1.
We make much use of Fatou's Lemma\index{Fatou, P. J. L.!Fatou's lemma|(}, that for non-negative
$f_n$,
$$\liminf_{n\to\infty}\int f_n\ge\int\liminf_{n\to\infty}f_n,$$
and also of its extended form: that if $f_n\le f$ and $f$ is integrable then
$$\limsup_{n\to\infty}\int f_n\le\int\limsup_{n\to\infty}f_n.$$
The latter may be deduced from the former by considering $f-f_n$.

\begin{proof}[Proof of Theorem \ref{t:L1}]We use the fact that for a random
variable $X$ with finite mean, and any constant $c$,
\begin{equation}E(X\one_{X>-c})
  =-cP(X>-c)+\int_{-c}^\infty P(X>x)\,dx,\label{e:Xone}\end{equation}
as may be proved by integrating $\int_{(-c,\infty)}x\,dP(X\le x)$ by parts.
We thus have, for $c>0$,
\begin{align*}ER_q&\le E(R_q\one_{R_q>-c})\\
  &=-cP(R_q>-c)+\int_{-c}^c P(R_q>x)\,dx+\int_c^\infty P(R_q>x)\,dx\\
  &=:A+B+C,\end{align*}
say. First, by the left-hand equality in (\ref{e:Rqtail}),
$\limsup_{q\to\infty}A=-cP(\alpha^{-1}Z\break>-c)$. Second, from the
right-hand equality in (\ref{e:Rqtail}), and the extended Fatou Lemma (take
the dominating integrable function to be 1),
$$\limsup_{q\to\infty}B\le\int_{-c}^c P(1+\alpha^{-1}Z>x)\,dx
  \le\int_{-c}^\infty P(1+\alpha^{-1}Z>x)\,dx.$$
Combining the bounds on $A$ and $B$ yields
\begin{align*}\limsup_{q\to\infty}(A+B)
  &\le-cP(1+\alpha^{-1}Z>-c)+\int_{-c}^\infty P(1+\alpha^{-1}Z>x)\,dx\\
    &\qquad{}+c\bigl(P(1+\alpha^{-1}Z>-c)-P(\alpha^{-1}Z>-c)\bigr)\\
  &=E\bigl((1+\alpha^{-1}Z)\one_{1+\alpha^{-1}Z>-c}\bigr)\\
    &\qquad{}+cP(-c-1<\alpha^{-1}Z\le-c)\\
  &<E\bigl((1+\alpha^{-1}Z)\one_{1+\alpha^{-1}Z>-c}\bigr)+cP(\alpha^{-1}Z\le-c).
\end{align*}

For the third upper bound, on $C$, we note (with an eye to the next proof as
well) that by Theorem \ref{t:Lp}, $K:=\sup_{q\in\N}E(\abs{R_q}^3)<\infty$. Then
for $x>0$, $P(R_q>x)\le K/x^3$, hence $C\le K/(2c^2)$. On combining this bound
with that on $A+B$ we gain an upper bound on $\limsup_{q\to\infty}ER_q$ that
converges to $E(1+\alpha^{-1}Z)$ as $c\to\infty$, concluding the proof of the
upper bound in (\ref{e:ERqbounds}).

For the lower bound we again use (\ref{e:Xone}), this time to write
\begin{align*}ER_q&=E(R_q\one_{R_q\le-c})+E(R_q\one_{R_q>-c})\\
  &=E(R_q\one_{R_q\le-c})-cP(R_q>-c)+\int_{-c}^\infty P(R_q>x)\,dx\\
  &=:\tilde A+\tilde B+\tilde C,\end{align*}
say. First, Fatou's Lemma and then the left-hand equality in
(\ref{e:Rqtail}) give
$$\liminf_{q\to\infty}\tilde C\ge\int_{-c}^\infty \liminf_{q\to\infty} P(R_q>x)\,dx
  =\int_{-c}^\infty P(\alpha^{-1}Z>x)\,dx.$$
Second,
$$\liminf_{q\to\infty}\tilde B=-c\limsup_{q\to\infty} P(R_q>-c)
  =-cP(1+\alpha^{-1}Z>-c),$$
this time by the right-hand equality in (\ref{e:Rqtail}). Combining, we find
that
\begin{align*}\liminf_{q\to\infty}(\tilde B+\tilde C)
  &\ge-cP(\alpha^{-1}Z>-c)+\int_{-c}^\infty P(\alpha^{-1}Z>x)\,dx\\
    &\quad{}-c\bigl(P(1+\alpha^{-1}Z>-c)-P(\alpha^{-1}Z>-c)\bigr)\\
  &=E(\alpha^{-1}Z\one_{\alpha^{-1}Z>-c})-cP(-c-1<\alpha^{-1}Z\le-c)\\
  &\ge E(\alpha^{-1}Z)-cP(\alpha^{-1}Z\le-c).\end{align*}

Finally, to put a lower bound on $\tilde A$ we may again use the `Markov
inequality'\index{Markov, A. A.!Markov inequality} method used above for $C$,
obtaining $\tilde A\ge-K/(2c^2)$. Combining this with the above, we gain a
lower bound on $\liminf_{q\to\infty}ER_q$ that converges to $E(\alpha^{-1}Z)$
as $c\to\infty$. We thus obtain the lower bound in (\ref{e:ERqbounds}).
\end{proof}

\begin{proof}[Proof of Lemma \ref{l:Rq2}]We use variants of the
decompositions in the previous proof. First, the upper bound. With $c>0$
fixed,
\begin{align*}E(R_q^2)&=E(R_q^2\one_{R_q>-c})+E(R_q^2\one_{R_q\le-c})\\
  &=c^2P(R_q>-c)+2\int_{-c}^\infty xP(R_q>x)\,dx+E(R_q^2\one_{R_q\le-c})\\
  &=c^2P(R_q>-c)+2\int_{-c}^0 xP(R_q>x)\,dx+2\int_0^c xP(R_q>x)\,dx\\
    &\qquad{}+2\int_c^\infty xP(R_q>x)\,dx+E(R_q^2\one_{R_q\le-c})\\
  &=:A+B_1+B_2+C+D,\end{align*}
say. By the right-hand equality in (\ref{e:Rqtail}),
$\limsup_{q\to\infty}A=c^2P(1+\alpha^{-1}Z>-c)$. By the left-hand
equality and Fatou's Lemma, followed by an integration by parts,
\begin{align*}\limsup_{q\to\infty}B_1&\le2\int_{-c}^0 xP(\alpha^{-1}Z>x)\,dx\\
  &=-c^2P(\alpha^{-1}Z>-c)+E\bigl((\alpha^{-1}Z)^2\one_{-c<\alpha^{-1}Z\le0}\bigr).
  \end{align*}
Combining,
\begin{align}\limsup_{q\to\infty}(A+B_1)
  &\le c^2P(-c-1<\alpha^{-1}Z\le-c)\notag \\
    &\qquad{}+E\bigl((\alpha^{-1}Z)^2\one_{-c<\alpha^{-1}Z\le0}\bigr)\notag \\
  &\le c^2P(\alpha^{-1}Z\le-c)+E\bigl((\alpha^{-1}Z)^2\one_{Z\le0}\bigr).
  \label{e:AB1}\end{align}

Next, by the right-hand equality in (\ref{e:Rqtail}), and the extended Fatou
Lemma,
\begin{align*}\limsup_{q\to\infty}B_2&\le2\int_0^c xP(1+\alpha^{-1}Z>x)\,dx\\
  &\le2\int_0^\infty xP(1+\alpha^{-1}Z>x)\,dx\\
  &=E\bigl((1+\alpha^{-1}Z)^2\one_{1+\alpha^{-1}Z>0}\bigr).\end{align*}
On combining this with (\ref{e:AB1}) and letting $c\to\infty$ we conclude that
\begin{multline*}\lim_{c\to\infty}\limsup_{q\to\infty}(A+B_1+B_2)\\
  \le E\bigl((\alpha^{-1}Z)^2\one_{Z\le0}
    +(1+\alpha^{-1}Z)^2\one_{1+\alpha^{-1}Z>0}\bigr).\end{multline*}
The upper bound in (\ref{e:Rq2}) will follow if we can show that
$\lim_{c\to\infty}\break\limsup_{q\to\infty}C=0$, and likewise for $D$. For
$C$ this follows by inserting into its defining formula the bound $P(R_q>x)\le
K/x^3$ developed in the proofs above, while for $D$ it follows from Theorem
\ref{t:Lp} via the \Index{uniform integrability} of the family $(R_q^2)_{q\in\N}$. The
upper bound in (\ref{e:Rq2}) is proved.

For the lower bound we fix $c>0$ and write
\begin{align*}E(R_q^2)&\ge E(R_q^2\one_{R_q>-c})\\
  &=c^2P(R_q>-c)+2\int_{-c}^\infty xP(R_q>x)\,dx\\
  &\ge c^2P(R_q>-c)+2\int_{-c}^0 xP(R_q>x)\,dx+2\int_0^c xP(R_q>x)\,dx.
  \end{align*}
In this right-hand side, use the left-hand equality in (\ref{e:Rqtail})
on the first term, use the right-hand equality and the extended Fatou Lemma
on the second term, and use the left-hand equality and Fatou's Lemma\index{Fatou, P. J. L.!Fatou's lemma|)}
on the third term, to give
\begin{align*}\liminf_{q\to\infty}E(R_q^2)&\ge c^2P(\alpha^{-1}Z>-c)
  +2\int_{-c}^0 xP(1+\alpha^{-1}Z>x)\,dx\\
  &\qquad{}+2\int_0^c xP(\alpha^{-1}Z>x)\,dx.\end{align*}
By two integrations by parts this becomes
\begin{align*}\liminf_{q\to\infty}E(R_q^2)&\ge c^2P(\alpha^{-1}Z>-c)
    -c^2P(1+\alpha^{-1}Z>-c)\\
    &\qquad{}+E\bigl((1+\alpha^{-1}Z)^2\one_{-c<1+\alpha^{-1}Z\le0}\bigr)\\
    &\qquad{}+c^2P(\alpha^{-1}Z>c)
    +E\bigl((\alpha^{-1}Z)^2\one_{0<\alpha^{-1}Z\le c}\bigr)\\
  &=-c^2P(-c-1<\alpha^{-1}Z\le-c)\\
    &\qquad{}+E\bigl((1+\alpha^{-1}Z)^2\one_{-c<1+\alpha^{-1}Z\le0}\bigr)\\
    &\qquad{}+c^2P(\alpha^{-1}Z>c)
    +E\bigl((\alpha^{-1}Z)^2\one_{0<\alpha^{-1}Z\le c}\bigr)\\
  &\ge-c^2P(\alpha^{-1}Z\le-c)
    +E\bigl((1+\alpha^{-1}Z)^2\one_{-c<1+\alpha^{-1}Z\le0}\bigr)\\
    &\qquad{}+E\bigl((\alpha^{-1}Z)^2\one_{0<\alpha^{-1}Z\le c}\bigr).\end{align*}
On letting $c\to\infty$ we obtain the lower bound in (\ref{e:Rq2}).
\end{proof}

\begin{proof}[Proof of Theorem \ref{t:Nqvar}]By Lemma \ref{l:Rq2},
\begin{align}&\limsup_{q\to\infty}E(R_q^2)\notag \\
  &\qquad{}\le E\bigl((\alpha^{-1}Z)^2\one_{Z\le0}\bigr)
    +E\bigl((1+\alpha^{-1}Z)^2\one_{0<1+\alpha^{-1}Z\le1}\bigr)\notag \\
    &\qquad\qquad{}+E\bigl((1+\alpha^{-1}Z)^2\one_{Z>0}\bigr)\notag \\
  &\qquad{}=E\bigl((\alpha^{-1}Z)^2\bigr)+\theta(\alpha)+P(Z>0)+\frac2\alpha
    E(Z\one_{Z>0}).\label{e:Rq2expand}\end{align}
Now $R_q=N_q-b_q$, so $\var N_q=\var R_q=E(R_q^2)-(ER_q)^2$. From
(\ref{e:ERqbounds}) we have $\liminf_{q\to\infty}ER_q\ge
E(\alpha^{-1}Z)=\gamma/\alpha>0$, so
$\liminf_{q\to\infty}(ER_q)^2\ge\bigl(E(\alpha^{-1}Z)\bigr)^2$. With
(\ref{e:Rq2expand}) this gives
$$\limsup_{q\to\infty}\var N_q\le\var(\alpha^{-1}Z)+\theta(\alpha)+P(Z>0)
  +\frac2\alpha E(Z\one_{Z>0}).$$
We have $\var Z=\pi^2/6$, while $P(Z>0)=1-e^{-1}$. Also
$E(Z\one_{Z>0})=\gamma-E(Z\one_{Z\le0})$, and
\begin{align*}-E(Z\one_{Z\le0})&=\int_{-\infty}^0(-z)e^{-z}\exp(-e^{-z})\,dz\\
  &=\int_1^\infty(\log t)e^{-t}\,dt=\int_1^\infty
  \frac{e^{-t}}t\,dt=E_1(1).\end{align*}
The bound
$$\limsup_{q\to\infty}\var N_q\le\frac{\pi^2}{6\alpha^2}+\theta(\alpha)+1-e^{-1}
    +\frac{2\bigl(\gamma+E_1(1)\bigr)}\alpha$$
follows.

For the lower bound, the lower bound in Lemma \ref{l:Rq2} may be written
\begin{align*}\liminf_{q\to\infty}E(R_q^2)
  &\ge E\bigl((1+\alpha^{-1}Z)^2\one_{Z\le0}\\
  &\qquad{}-E\bigl((1+\alpha^{-1}Z)^2\one_{0<1+\alpha^{-1}Z\le1}\bigr)
  +E\bigl((\alpha^{-1}Z)^2\one_{Z>0}\bigr)\\
  &=E\bigl((\alpha^{-1}Z)^2\bigr)-\theta(\alpha)+P(Z\le0)
    +2E(\alpha^{-1}Z\one_{Z\le0}).
  \end{align*}
From (\ref{e:ERqbounds}) we have $0<\limsup_{q\to\infty}ER_q\le
E(1+\alpha^{-1}Z)$, so $\limsup_{q\to\infty}\break
(ER_q)^2\le1+2E(\alpha^{-1}Z)+\bigl(E(\alpha^{-1}Z)\bigr)^2$,
which with the above gives
\begin{multline*}\liminf_{q\to\infty}\var R_q\\
  \ge\var(\alpha^{-1}Z)-\theta(\alpha)-1+P(Z\le0)
  -2\alpha^{-1}\bigl(EZ-E(Z\one_{Z\le0})\bigr).\end{multline*}
Thus, since $\var N_q=\var R_q$,
$$\liminf_{q\to\infty}\var N_q
  \ge\frac{\pi^2}{6\alpha^2}-\theta(\alpha)-1+e^{-1}
    -\frac{2(\gamma+E_1(1))}\alpha,$$
which is the required lower bound on $\liminf_{q\to\infty}\var N_q$ and
completes the proof.
\end{proof}

\section{Numerical results}\label{s:results}
\index{Matlab@{\it Matlab}}\emph{Matlab} and
\emph{Pascal}\index{Pascal, B.!Pascal programming language@\emph{Pascal} programming language} were used to evaluate $EN_q$ for different values
of $a$ and $q$. Fig.\ \ref{f:ENlowq} shows values for $EN_q$ for different
values of $a$ for tests with up to 20 questions. For example, for a test with
10 alternatives for each question $EN_q$ ranges from 29 when there is one
question in the test to only 56 when there are 20 questions. Contrast this
with the total number of possible tests, which increases from 10 to $10^{20}$
in this range.

\begin{figure}[!ht]
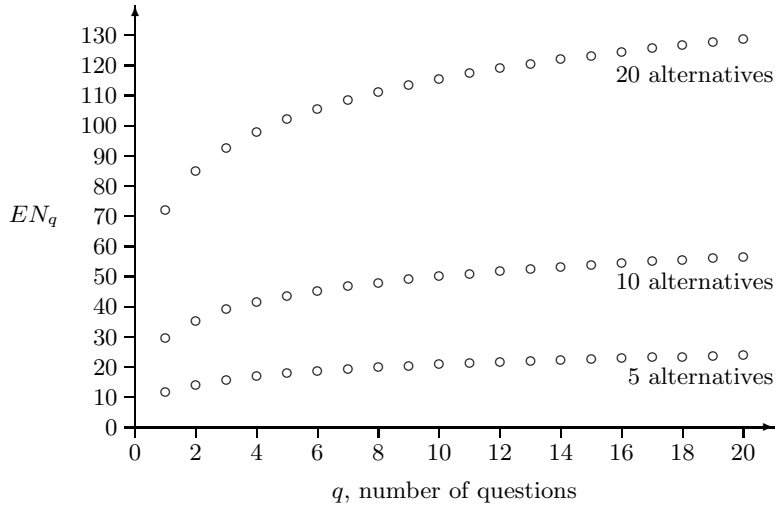

$$\vbox{\small
\beginpicture
\normalgraphs
\setcoordinatesystem units <4mm,.4mm>
\unitlength=4mm
\setplotarea x from 0 to 21, y from 0 to 140
\axis bottom invisible label {$q$, number of questions}
  ticks numbered from 0 to 20 by 2 /
\axis left invisible label {$EN_q$} ticks numbered from 0 to 130 by 10 /
\put {\vector(1,0){21}} [Bl] at 0 0
\put {\vector(0,1){14}} [Bl] at 0 0
\accountingoff
\multiput {$\circ$} at 1 11.4  2 14  3 15.7  4 16.9  5 17.8  6 18.6
  7 19.2  8 19.8  9 20.3  10 20.8  11 21.2  12 21.6  13 21.9  14 22.2
  15 22.5  16 22.8  17 23.1  18 23.3  19 23.6  20 23.8 /
\put {5 alternatives} [tr] <0pt,-1.5mm> at 21 23.8
\multiput {$\circ$} at 1 29.3  2 35.2  3 38.9  4 41.5  5 43.5  6 45.2
  7 46.6  8 47.8  9 48.9  10 49.9  11 50.8  12 51.6  13 52.3  14 53
  15 53.7  16 54.3  17 54.9  18 55.4  19 55.9  20 56.4 /
\put {10 alternatives} [tr] <0pt,-2mm> at 21 56.4
\multiput {$\circ$} at 1 72  2 84.7  3 92.3  4 97.8  5 102  6 105.5
  7 108.5  8 111  9 113.3  10 115.3  11 117.1  12 118.8  13 120.4  14 121.8
  15 123.1  16 124.4  17 125.6  18 126.7  19 127.7  20 128.7 /
\put {20 alternatives} [tr] <0pt,-3.5mm> at 21 128.7
\accountingon
\endpicture
}$$
	\caption{{\protect \small $EN_q$, the expected number of tests that
          need to be generated in order for all questions to have appeared
          at least once, for tests with up to 20 questions and 5, 10, and
          20 alternatives for each question.}}
	\label{f:ENlowq}
\end{figure}

These results led the authors to extend the investigation to consider tests
containing up to 200 questions. Fig.\ \ref{f:ENhighq} demonstrates that, as
the number of questions in a test is increased, the average number of tests
required in order for all possible questions to have appeared increases quite
slowly. In a 200-question test with 10 alternatives for each question, there
are $10^{200}$ different possible tests and a total bank\index{question bank} of 2000 questions;
however, on average all questions will have appeared at least once by the time
only 78 tests have been generated. Table \ref{ta:ENq} summarises the results
from Fig.\ \ref{f:ENhighq}, giving $EN_q$ for different values of $a$ and $q$.

\begin{table}[!ht]\footnotesize
  \caption{Values of $EN_q$ for various values of $a$ and $q$.}
  \label{ta:ENq}
\begin{tabular}{@{}cccccccc@{}}\hline
Number of alternatives&\multicolumn{7}{|c}{Number of questions in test $(q)$}\\
for each question$(a)$&$1$&$5$&$10$&$20$&$50$&$100$&$200$\\\hline
$5$&$11\cd4$&$17\cd8$&$20\cd8$&$23\cd8$&$27\cd9$&$31\cd0$&$34\cd1$\\
$10$&$29\cd3$&$43\cd5$&$49\cd9$&$56\cd4$&$65\cd0$&$71\cd6$&$78\cd1$\\
$20$&$72\cd0$&$102\cd0$&$115\cd3$&$128\cd7$&$146\cd5$&$160\cd0$&$173\cd5$\\\hline
\end{tabular}
\end{table}

\begin{figure}[!ht]
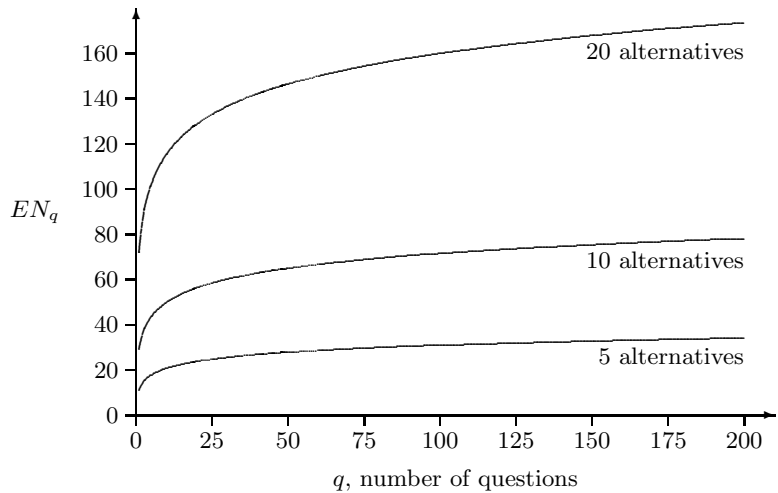

$$\vbox{\small
\beginpicture
\normalgraphs
\setcoordinatesystem units <.4mm,.3mm>
\unitlength=.4mm
\setplotarea x from 0 to 210, y from 0 to 180
\axis bottom invisible label {$q$, number of questions}
  ticks numbered from 0 to 200 by 25 /
\axis left invisible label {$EN_q$} ticks numbered from 0 to 160 by 20 /
\put {\vector(1,0){210}} [Bl] at 0 0
\put {\vector(0,1){135}} [Bl] at 0 0
\put {5 alternatives} [tr] <0pt,-1mm> at 200 33.5
\put {10 alternatives} [tr] <0pt,-1.5mm> at 200 77.6
\put {20 alternatives} [tr] <0pt,-2.5mm> at 200 173
\setquadratic
\plot 1 11.4  2 14  3 15.7  4 16.9  5 17.8  6 18.6  7 19.2  8 19.8  9 20.3
  10 20.8  11 21.2  12 21.6  13 21.9  14 22.2  15 22.5  16 22.8  17 23.1
  18 23.3  19 23.6  20 23.8  25 24.8  30 25.6  35 26.3  40 26.9  42 27.1
  44 27.3  45 27.4  46 27.5  48 27.7  50 27.9  60 28.7  80 30  100 31
  150 32.8  200 34.1 /
\plot 1 29.3  2 35.2  3 38.9  4 41.5  5 43.5  6 45.2  7 46.6  8 47.8  9 48.9
  10 49.9  11 50.8  12 51.6  13 52.3  14 53  15 53.7  16 54.3  17 54.9
  18 55.4  19 55.9  20 56.4  25 58.5  30 60.2  35 61.6  40 62.9  42 63.4
  44 63.8  45 64  46 64.2  48 64.6  50 65  60 66.7  80 69.5  100 71.6
  150 75.4  200 78.1 /
\plot 1 72  2 84.7  3 92.3  4 97.8  5 102  6 105.5  7 108.5  8 111  9 113.3
  10 115.3  11 117.1  12 118.8  13 120.4  14 121.8  15 123.1  16 124.4
  17 125.6  18 126.7  19 127.7  20 128.7  25 133  30 136.6  35 139.6
  40 142.2  42 143.1  44 144  45 144.4  46 144.9  48 145.7  50 146.5
  60 150  80 155.6  100 160  150 167.9  200 173.5 /
\endpicture
}$$
	\caption{{\protect \small $EN_q$, the average number of tests that
          need to be generated in order for all questions to have appeared at
          least once, for tests with up to 200 questions and 5, 10, and 20
          alternatives for each question.}}
	\label{f:ENhighq}
\end{figure}

\section{Discussion}\label{s:discuss}
The asymptotics concern the behaviour of the random variable $N_q$,
defined in \S\ref{s:howmany}, as the number of questions, $q$, grows. There is
also dependence on $a$, the number of alternative answers per question in the
multiple choice, but we regard $a$ as fixed; it is any integer at least 2, and
we set
$$\alpha:=\log\Bigl(\frac a{a-1}\Bigr),$$
so $\alpha>0$. Theorem \ref{t:noconv} first says that $N_q$ cannot be centred
and normed so that its distribution properly converges (one could get
convergence to 0, of course, just by heavy norming)\undex{norming}. However it then says that
by centring (translation) alone, $N_q$ comes very close to looking like the
random variable $Z/\alpha$, where $Z$ has the Gumbel
distribution\index{Gumbel, E. J.!Gumbel distribution}. The
difference is a `wobble' of between 0 and 1 in the limit; persistence of
discreteness is responsible for this.

Theorem \ref{t:L1} establishes that the expected value of $N_q$ behaves
accordingly, growing like $\alpha^{-1}\log q$. More exactly, after \Undex{centring} by
$b_q:=\alpha^{-1}\log(aq)$ it differs from $\alpha^{-1}EZ$ by a number between
0 and 1 in the limit. Table \ref{ta:ENqlim} gives values of
$b_q+\alpha^{-1}EZ$ for different values of $a$ and $q$.
\begin{table}[!ht]\footnotesize
  \caption{Values of $b_q+\alpha^{-1}EZ$ for various values of $a$ and $q$.}
  \label{ta:ENqlim}
\begin{tabular}{@{}c|ccccccc@{}}\hline
Number of alternatives&\multicolumn{7}{|c}{Number of questions in test
  $(q)$}\\
for each question $(a)$&$1$&$5$&$10$&$20$&$50$&$100$&$200$\\\hline
$5$&$9\cd8$&$17\cd0$&$20\cd1$&$23\cd2$&$27\cd3$&$30\cd4$&$33\cd5$\\
$10$&$27\cd3$&$42\cd6$&$49\cd2$&$55\cd8$&$64\cd5$&$71\cd0$&$77\cd6$\\
$20$&$68\cd7$&$101\cd0$&$114\cd5$&$128\cd1$&$145\cd9$&$159\cd4$&$173\cd0$\\\hline
\end{tabular}
\end{table}
For $q\ge20$ the actual values of $EN_q$ in the
previous table exceed these by $0\cd5$--$0\cd6$, exactly as Theorem \ref{t:L1}
predicts.

What about the variance of $N_q$ as $q$ grows? Theorem \ref{t:Nqvar} says that
it does not tend to infinity, but is trapped as $q\to\infty$ between bounds
that do not depend on $q$. The precision is pleasing, given that $N_q$ does
not converge, in any sense. The asymptotic bounds on the variance of $N_q$,
$\var N_q$, are $\frac{\pi^2}{6\alpha^2}\pm\Delta$ where $\Delta$ is a strange
jumble of constants:
$$\Delta=\theta(\alpha)+1-e^{-1}
    +\frac{2\bigl(\gamma+E_1(1)\bigr)}\alpha$$
(the bounds are not claimed to be sharp).

\begin{table}[!ht]\small
\begin{minipage}{80mm}
  \caption{Asymptotic bounds on the standard deviation of $N_q$.}
  \label{ta:sd}
\end{minipage}
\begin{tabular}{@{}l|cccccc@{}}\hline
$a$&2&3&4&5&10&20\\\hline
Min s.d.&$0\cd641$&$2\cd323$&$3\cd697$&$5\cd024$&$11\cd507$&$24\cd362$\\
Max s.d.&$2\cd537$&$3\cd823$&$5\cd107$&$6\cd390$&$12\cd804$&$25\cd630$\\\hline
\end{tabular}
\end{table}

The amount of variability can be better appreciated through the standard
deviation. The asymptotic bounds on the standard deviation of $N_q$ are
$$\sqrt{\frac{\pi^2}{6\alpha^2}\pm\Delta},$$
and some values for these are in Table \ref{ta:sd}.
The lower bound is non-trivial, i.e.\ positive, in each case.

\paragraph{Acknowledgements}
We are grateful to Dave Pidcock\index{Pidcock, D.}, a colleague in the
Mathematics Education Centre at Loughborough University, for raising the query
in the first place. As a member of staff using computer-based tests to assess
students, he was concerned about this issue from a practical viewpoint. That
led RC to post a query on \Index{Allstat}. CMG was not the only person to
respond to the query, and we also acknowledge the others who responded,
particularly Simon Bond\index{BondS@Bond, S.}.\index{computer based test@computer-based test|)}

\end{document}